\documentclass[12pt,a4paper]{amsart}
\usepackage{a4wide}
\usepackage{amsfonts,amsthm,amsmath,amssymb,amscd}
\usepackage{graphicx}

\usepackage{hyperref}

\theoremstyle{plain}
\newtheorem{lem}{Lemma}
\newtheorem{prop}[lem]{Proposition}
\newtheorem{thm}[lem]{Theorem}
\newtheorem*{thm*}{Theorem}
\newtheorem{cor}[lem]{Corollary}
\newtheorem*{cor*}{Corollary}

\theoremstyle{definition}

\newtheorem*{defn*}{Definition}

\newtheorem*{ex*}{Example}

\newtheorem*{rem*}{Remark}

\theoremstyle{remark}

\DeclareMathOperator{\inter}{int}
\DeclareMathOperator{\diam}{diam}

\DeclareMathOperator{\length}{length}
\newcommand{\C}{\mathbb C}

\newcommand{\R}{\mathbb R}
\newcommand{\Z}{\mathbb Z}

\newcommand{\B}{\mathcal B}

\newcommand{\LL}{\mathcal L}

\newcommand{\bd}{\partial}

\renewcommand{\Re}{\textup{Re}}

\begin{document}

\title[Hyperbolic dimension of Julia sets]{Hyperbolic dimension of
Julia sets of meromorphic maps with logarithmic tracts}

\date{\today}

\author{Krzysztof Bara\'nski}
\address{Institute of Mathematics, University of Warsaw,
ul.~Banacha~2, 02-097 Warszawa, Poland}
\email{baranski@mimuw.edu.pl}

\author{Bogus{\l}awa Karpi\'nska}
\address{Faculty of Mathematics and Information Science, Warsaw
University of Technology, Pl.~Politechniki~1, 00-661 Warszawa, Poland}
\email{bkarpin@impan.gov.pl}

\author{Anna Zdunik}
\address{Institute of Mathematics, University of Warsaw,
ul.~Banacha~2, 02-097 Warszawa, Poland}
\email{A.Zdunik@mimuw.edu.pl}

\thanks{Research supported by Polish MNiSW Grant N N201 0234 33 and
EU FP6 Marie Curie ToK Programme SPADE2.}
\subjclass[2000]{Primary 37F10, 37F35, 30D40, 28A80.}

\begin{abstract} We prove that for meromorphic maps with logarithmic
tracts (e.g. entire or meromorphic maps with a finite number of
poles from class $\B$), the Julia set contains a compact invariant hyperbolic
Cantor set of Hausdorff dimension greater than $1$. Hence, the
hyperbolic dimension of the Julia set is greater than $1$.
\end{abstract}

\maketitle

Following \cite{BRS}, we say that an unbounded domain $T \in \C$ with
piecewise smooth boundary and unbounded complement is a logarithmic
tract for a continuous function $f: \overline{T} \to \C$, if there
exists $R_0 > 0$, such that $|f(z)| > R_0$ for every $z \in T$, $|f(z)| =
R_0$ for every $z \in \bd T$ and $f$  on $T$ is a holomorphic universal
covering of $\{z: |z| > R_0\}$. Note that every meromorphic map with
logarithmic singularity over infinity has a logarithmic tract. In
particular, this holds for transcendental meromorphic maps with a
finite number of poles (in particular transcendental entire maps) from
class $\B$. Recall that the class $\B$ consists of maps $f$, for which
the set of singularities of $f^{-1}$ is bounded.

The hyperbolic dimension of the Julia set $J(f)$ is defined
as the supremum of the Hausdorff dimensions (denoted $\dim_H$) of all
conformal expanding
Cantor repellers contained in $J(f)$. Recall that in this setting a
conformal expanding Cantor repeller is a compact invariant Cantor set
$X \subset J(f)$, such that $(f^k)'|_X > 1$ for some $k >
0$. Obviously, the hyperbolic dimension of $J(f)$ is not greater than its
Hausdorff dimension. However, it can be strictly smaller, see \cite{UZ}.

In this note we prove the following.

\begin{thm*} The hyperbolic dimension of the Julia set of a
meromorphic map with a logarithmic tract is greater than $1$. In
particular, the Hausdorff dimension of the set of points with bounded
orbits in the Julia set is greater than $1$.
\end{thm*}

It is known that the Julia set of any entire transcendental map
contains non-degenerate continua (see \cite{Ba}) and hence 
$\dim_H(J(f)) \geq 1$. On the other hand, it can be
arbitrarily close to $1$ (see \cite{S1}). 
In \cite{S2}, Stallard proved a remarkable result stating that
the Hausdorff dimension of the Julia set 
of entire maps from the class $\B$ is greater than $1$.
The result was extended by
Rippon and Stallard in \cite{RS} to meromorphic maps with
finitely many poles, and by Bergweiler, Rippon and Stallard in
\cite{BRS} to maps with logarithmic tracts.

The proof of our theorem uses thermodynamic formalism. Note that the
result was proved in \cite{UZ} for the specific family of exponential maps.

Let $f$ be a meromorphic map with a logarithmic tract $T$. It is
convenient to look at $f$ in the logarithmic coordinates, introduced
in \cite{EL}, which are defined as follows. Changing coordinates by
translation, we can assume that
$0 \notin \overline{T}$. Moreover, enlarging $R_0$ in the
definition of the logarithmic tract $T$, we can assume that $\bd T$ is
homeomorphic to a straight line, so that $\overline{T}$ is simply
connected and $\bd T \cup \{\infty\}$ is a Jordan curve in the Riemann
sphere. Since $f$ on
$T$ is a universal cover of $\{z: |z| > R_0\}$, we can lift $f$ by the
branches of logarithm to a map
\[
F: \bigcup_{s \in \Z} \overline{L_s} \to \overline{H},
\]
where $L_s = L_0 + 2 \pi is$ are unbounded simply connected domains in
$\C$ (lifted tracts), such that $\bd L_s \cup \{\infty\}$ is a Jordan
curve and $H = \{z: \Re(z) > \ln R_0\}$. We have
\[
\exp\circ F = f \circ \exp,
\]
in particular, $F$ is periodic with period $2 \pi i$. Note that $F$ is
conformal on each $L_s$ and maps $\overline{L_s}$ homeomorphically
onto $\overline{H}$. Let $F^{-1}_s$ be the inverse branch of $F$
mapping $\overline{H}$ onto $\overline{L_s}$. Note that $F^{-1}_s =
F^{-1}_0 + 2 \pi is$.

\begin{lem} \label{lem:F'<} Let $\varepsilon > 0$. Then there exist
arbitrarily large $x > 0$, such that for every $s \in \Z$,
\[
|(F^{-1}_s)'(x)| > \frac{1}{x^{1 + \varepsilon}}.
\]
\end{lem}
\begin{proof} Let $x > 0$. Since $F^{-1}_s(x) \to \infty$ as $x \to
+\infty$, the curve $F^{-1}_s([\ln R_0, +\infty))$ has infinite
length. Hence, the integral
\[
\int_{\ln R_0}^{+\infty} |(F^{-1}_s)'(x)|\; dx
\]
is diverging near infinity, which easily shows the lemma.
\end{proof}

Using the Koebe one-quarter theorem, we  obtain the result showed by Eremenko and Lyubich as Lemma~1 in \cite{EL}.

\begin{lem} \label{lem:F'>} For every $z \in \bigcup_{s \in \Z}
L_s$, we have $|F'(z)| > \frac{1}{4\pi} (\Re(F(z)) - \ln R_0)$.
\qed
\end{lem}

\begin{cor} \label{cor:Re<} There exist $a, b > 0$, such that for
every $x \geq \ln R_0 + 1$ and $s \in \Z$,
\[
\Re(F^{-1}_s(x)) < a \ln x + b.
\]
\end{cor}
\begin{proof} By Lemma~\ref{lem:F'>},
\[
|F^{-1}_s(x) - F^{-1}_s(1 + \ln R_0)| < 4 \pi \int_{1 + \ln R_0}^x
\frac{dt}{t - \ln R_0}  = 4 \pi \ln (x - \ln R_0).
\]
Hence, $\Re(F^{-1}_s(x)) < 4 \pi \ln (x - \ln R_0) + c_0$, where $c_0 =
\Re(F^{-1}_s(1 + \ln R_0))$ is independent of $s$, which implies the
assertion.
\end{proof}

We need one more simple observation:
\begin{lem}\label{Re>}
If $z \in \overline{H}$ and $z \to\infty$, then $\Re (F^{-1}_s(z))\to\infty$.
\end{lem}
\begin{proof}
It is sufficient to notice that for $M > 0$ the set
$\overline{L_s}\cap \{z: \Re (z)\le M\}$ is compact.
\end{proof}

To prove our theorem, we first construct a suitable Cantor repeller in the
logarithmic coordinates.

\begin{prop} \label{prop:J_B} There exists a compact set $J_B\subset
\bigcup_{s \in \Z}
L_s$ with $\dim_H(J_B)>1$, such that $F(J_B)\subset \bigcup_{s \in \Z}
L_s$ and $F^2(J_B)=J_B$.
In particular, the forward trajectory of every point in $J_B$ is well
defined and bounded.
Moreover, $(F^n)'(z)\to\infty$ as $n \to \infty$ for every $z\in J_B$.
\end{prop}

To prove this proposition, note first that by Lemma~\ref{lem:F'<},
there exists an arbitrarily large $R > 0$, such that
\begin{equation} \label{eq:eps}
|(F^{-1}_s)'(R)|> 1/R^{1+\varepsilon}
\end{equation}
for a small fixed $\varepsilon > 0$. Recall that all preimages
\[
v_s =F^{-1}_s(R)
\]
are
located on the same vertical line
\[
\ell=\{v:\Re (v)=r\}
\]
and, by Corollary~\ref{cor:Re<}, we have
\[
r=\Re(v_s) < a \ln R + b\ll R
\]
for large $R$. On the other hand, given $D>0$, and using Lemma~\ref{Re>}, we can
require that
\begin{equation}\label{re w}
r=\Re (v_s)>\ln R_0+ 2D,
\end{equation}
so, in particular, the line $\ell$ is contained in $H$.

Now let us consider all components of $F^{-1}(\ell)$, i.e. the sets
\[
\ell_u=F^{-1}_u(\ell)
\]
for $u \in \Z$. Obviously, $\ell_u\subset L_u$ and again, using
Corollary~\ref{cor:Re<}, we conclude that
\begin{equation} \label{eq:l_u}
\inf_{v\in \ell_u}\Re (v)< a\ln r+b\ll r\ll R,
\end{equation}
if $R$ is large. Denote
\[
v_{u,s}=F^{-1}_u(v_s)=F^{-1}_u \circ F^{-1}_s(R).
\]
and consider the square
\[
Q=\left[\frac{R}{2},\frac{3R}{2}\right]\times
\left[-\frac{R}{2},\frac{R}{2}\right]
\]
(so that $R$ is the center of $Q$ and the sides of $Q$ have length
$R$). See Figure~\ref{fig:square}.

\begin{figure}[!ht]
\begin{center}
\includegraphics*[height=8cm]{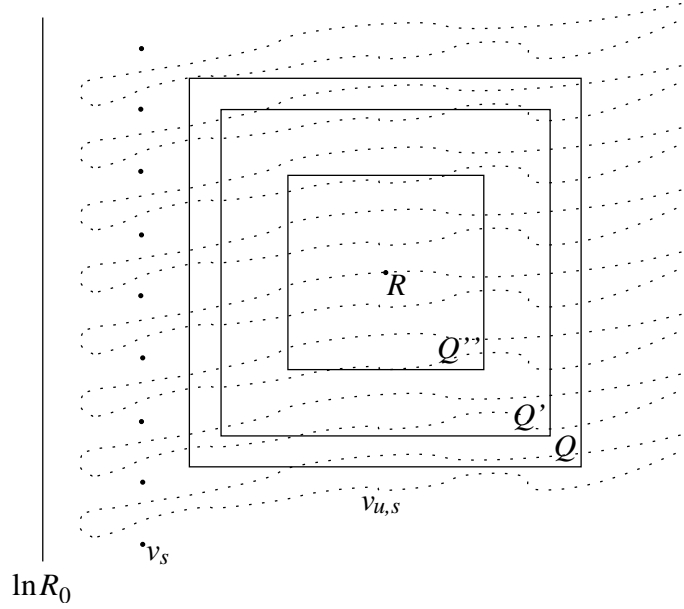}
\caption{The squares $Q, Q', Q''$.}
\label{fig:square}
\end{center}
\end{figure}

Next, let us consider the image of $Q$ under $F^{-1}_s$. By
Lemma~\ref{lem:F'>},
\begin{equation} \label{eq:4pi}
|(F^{-1}_s)'(R)| <\frac{4\pi}{R-\ln R_0},
\end{equation}
and by the Koebe Distortion Theorem (see e.g. \cite{CG}), the distortion
of $F_s^{-1}$ on $Q$ is universally bounded. We can thus write for
every $z\in Q$,
$$|(F^{-1}_s)'(z)| < C \frac{4\pi}{R-\ln R_0},$$
where $C > 1$ is the distortion constant. Since $\diam (Q)=\sqrt{2}R$, we have
$$\diam(F^{-1}_s(Q))\le\diam(Q) \frac{4C\pi}{R-\ln
R_0}=\frac{4\sqrt{2}\pi CR}{R-\ln R_0} <
5\sqrt{2}\pi C$$
for large $R$. Hence, by \eqref{re w} applied for $D = 5\sqrt{2}\pi C$, we get
$F^{-1}_s(Q)\subset H$, which implies that the inverse branches
$F^{-1}_u$ are defined in $F^{-1}_s(Q)$. Moreover, by
Lemma~\ref{lem:F'>} and \eqref{re w},
\[
|(F^{-1}_u)'(v_s)|< \frac{4\pi}{r - \ln R_0} < \frac{2\pi}{D} < 1,
\]
so using \eqref{eq:4pi}, we have the following estimate:
$$|(F^{-1}_u\circ F^{-1}_s)'(R)|<\frac{4\pi}{R-\ln R_0},$$
which implies, again by the Koebe Distortion Theorem, that
$$|(F^{-1}_u\circ F^{-1}_s)'(z)|< \frac{4\pi C}{R-\ln R_0}$$ for all $z\in Q$.
We conclude that, denoting
\[
Q_{u,s}=F^{-1}_u \circ F^{-1}_s(Q),
\]
we have
\begin{equation} \label{eq:diamQ_us}
\diam (Q_{u,s})<\diam(Q)\frac{4\pi C}{R-\ln R_0}
 = \frac{4\sqrt{2}\pi CR}{R-\ln R_0} < D.
\end{equation}

Now we consider the set $G$ of indices $(u,s)$ defined as follows:
$$G=\{(u,s)\in \mathbb{Z}^2: Q_{u,s} \subset Q\}.$$
We shall prove

\begin{lem}\label{pressure}
There exists a constant $C_1 > 0$, such that for every $z\in Q$ we have
\begin{equation}
\sum_{(u,s)\in G} |(F^{-1}_u\circ F^{-1}_s)'(z)|>C_1 R^{1-\varepsilon}
\end{equation}
\end{lem}

\begin{proof}
Let $Q''\subset Q'\subset Q$ be the following squares centered at $R$:
$$Q'=\left[\frac{R}{2} + D, \frac{3R}{2} - D\right] \times
\left[-\frac{R}{2} + D, \frac{R}{2} -D\right], \quad
Q''=\left[\frac{3R}{4}, \frac{5R}{4}\right] \times
\left[-\frac{R}{4}, \frac{R}{4}\right].$$
See Figure~\ref{fig:square}.

Let $\LL$ be the family of all curves $\ell_u$
intersecting $Q''$. For every $\ell_u \in \LL $, we denote by
$\ell_{u,Q'}$ a connected component of $\ell_u \cap Q'$, such that
$\ell_{u,Q'}\cap Q'' \neq\emptyset$. Note that every
$\ell_{u,Q'}$ intersects the boundaries of $Q'$ and $Q''$, so
\begin{equation} \label{eq:length>}
\length(\ell_{u,Q'}) \geq \frac{R}{4} - D.
\end{equation}
Now, consider the subset $G_u\subset G$ consisting of all $s\in
\mathbb{Z}$, such that $v_{u,s} \in Q'$. Note that by
\eqref{eq:diamQ_us},
\[
\{(u, s): s \in G_u\} \subset G.
\]
Moreover, $\ell_{u,Q'} \subset F^{-1}_u(\bigcup_{s\in G_u} B(v_s, 2\pi))$
and the distortion of $F^{-1}_u$ is uniformly bounded on every
ball $B(v_s, 2\pi)$. Hence,
\begin{equation} \label{eq:length<}
\length(\ell_{u,Q'}) < C_2 \sum_{s\in G_u} |(F^{-1}_u)'(v_s)|
\end{equation}
for some constant $C_2 > 0$. Using \eqref{eq:eps},
\eqref{eq:length>}, \eqref{eq:length<} and the fact that the distortion
of $F^{-1}_u\circ F^{-1}_s$ is universally bounded on $Q$, we get
\begin{multline*}
\sum_{s\in G_u} |(F^{-1}_u\circ F^{-1}_s)'(z)|>
\frac{1}{C} \sum_{s\in G_u} |(F^{-1}_u\circ F^{-1}_s)'(R)|
\\> \frac{1}{CR^{1 + \varepsilon}} \sum_{s\in
G_u}|(F^{-1}_u)'(v_s)|>\frac{1}{CC_2}\frac{R/4 - D}{R^{1 +
\varepsilon}} > \frac{C_3}{R^\varepsilon},
\end{multline*}
for some constant $C_3 > 0$, if $R$ is large. Moreover, using
\eqref{eq:l_u}, Lemma~\ref{Re>} and the fact $\ell_u
= \ell_0 + 2 \pi i$, we see that the family $\LL$ contains at least
$R/(4\pi)$ curves, so finally we get
\[
\sum_{(u,s)\in G} |(F^{-1}_u\circ F^{-1}_s)'(z)|\ge \sum_{\{u: \ell_u
\in \LL\}}\sum_{s\in
G_u} |(F^{-1}_u\circ F^{-1}_s)'(z)| >  \frac{C_3}{4\pi} R^{1 - \varepsilon}.
\]
\end{proof}

Now, we are in a position to prove Proposition~\ref{prop:J_B}.
Let us consider the collection of open, simply connected sets
$\inter Q_{u, s}$, where $(s,t)\in G$. It follows from our
assumptions that their closures are pairwise disjoint and, by
\eqref{eq:diamQ_us} and the definition of $G$,
$Q_{u, s}\subset Q$. In this way we get a conformal
Iterated Function System, formed by strictly contracting, conformal
maps $F^{-1}_u\circ F^{-1}_s: Q\to Q$. Such a system has a unique
compact invariant set $J_B$. Since the $Q_{u, s}$ are pairwise disjoint,
$J_B$ is a Cantor set. By construction, $F(J_B)\subset \bigcup_{s \in \Z}
L_s$, $F^2(J_B)=J_B$ and $(F^n)'(z) \to \infty$ for every $z \in J_B$.
It is well known that the Hausdorff dimension
of $J_B$ is determined as the unique $t\in \R$, such that
$P(t)=0$, where $P(t)$ is so-called pressure function (see
e.g. \cite{PU}). Let us recall the definition of the pressure
function in our setting:
$$P(t)=\lim_{n\to\infty}\frac{1}{n}\log\sum_{g^n}\|(g^n)'\|^t.$$
Here, we sum up over all possible compositions $g^n=
g_{i_1}\circ\dots\circ
g_{i_n}$, where $g_{i_j}=F^{-1}_u\circ F^{-1}_s$ for some $(u,s)\in G$.

Since the system is expanding, the function $t\mapsto P(t)$ is
strictly decreasing and in order to prove $\dim_H(J_B)>1$ it is enough
to show that $P(1)>0$.
Using Lemma~\ref{pressure} we can estimate the pressure $P(1)$ as follows:
\begin{multline*}
P(1)=\lim_{n\to\infty}\frac{1}{n}\log\sum_{g^n}\|(g^n)'\|
\\\ge \lim_{n\to\infty}\frac{1}{n}\log\left( \inf_{z\in
Q}\sum_{(u,s)\in G} |(F^{-1}_u\circ F^{-1}_s)'(z)|\right )^n
\ge
\log (C_1R^{1-\varepsilon})> 0
\end{multline*}
for $\varepsilon < 1$ and $R$ sufficiently large. Therefore, $\dim_H(J_B)>1$.
This ends the proof of Proposition~\ref{prop:J_B}.

Proposition~\ref{prop:J_B} immediately implies the following theorem, which
proves our main result.

\begin{thm}
There exists a conformal expanding Cantor repeller $X \subset J(f)$
with the Hausdorff dimension greater than $1$.
\end{thm}
\begin{proof}
Let $X =\exp (J_B \cup F(J_B))$. Since $\exp\circ F=f\circ \exp$ and
$J_B \cup F(J_B)$ is $F$-invariant, $X$ is $f$-invariant. Moreover,
since $\exp$ is a smooth covering,
$\dim_H(X)=\dim_H(J_B)>1$. Finally, since $\lim_{n\to\infty}
(F^n)'(z)=\infty$ for $z\in J_B$,
we conclude that also $\lim_{n\to\infty} (f^n)'(z)=\infty$ for $z\in
X$. On the other hand, the trajectory $f^n(z)$ is bounded.
Therefore, the family of iterates $\{f^n\}$ cannot be normal in any
neighbourhood of $z \in X$. Consequently, $X\subset J(f)$.
\end{proof}

\end{document}